\documentclass[11pt,reqno,a4paper]{amsart}
\usepackage[toc,page]{appendix}

\usepackage[utf8]{inputenc}
\usepackage{esint}
\usepackage{MnSymbol}
\usepackage{enumitem}
\usepackage[english]{babel}
\usepackage{amsmath}
\usepackage{mathtools}
\usepackage{thm-restate}
\usepackage{comment}
\usepackage{bm}
\usepackage{tikz}
\usetikzlibrary{patterns}
\usetikzlibrary{arrows.meta}

\usepackage{xcolor}

\newtheorem{theorem}{Theorem}
\newtheorem*{theorem*}{Theorem}
\newtheorem{proposition}{Proposition}[section]
\newtheorem{lemma}[proposition]{Lemma}

\theoremstyle{definition}
\newtheorem{definition}[proposition]{Definition}
\newtheorem{remark}[proposition]{Remark}

\numberwithin{equation}{section}

\usepackage[
colorlinks,
pdfpagelabels,
pdfstartview = FitH,
bookmarksopen = true,
bookmarksnumbered = true,
linkcolor = blue,
plainpages = false,
hypertexnames = false,
citecolor = red] {hyperref}

\hypersetup{
linktoc=page
}

\usepackage[capitalize]{cleveref}
\crefname{enumi}{}{}
\crefname{equation}{}{}

\makeatletter
\def\@tocline#1#2#3#4#5#6#7{\relax
\ifnum #1>\c@tocdepth 
\else
\par \addpenalty\@secpenalty\addvspace{#2}%
\begingroup \hyphenpenalty\@M
\@ifempty{#4}{%
\@tempdima\csname r@tocindent\number#1\endcsname\relax
}{%
\@tempdima#4\relax
}%
\parindent\z@ \leftskip#3\relax \advance\leftskip\@tempdima\relax
\rightskip\@pnumwidth plus4em \parfillskip-\@pnumwidth
#5\leavevmode\hskip-\@tempdima
\ifcase #1
\or\or \hskip 1em \or \hskip 2em \else \hskip 3em \fi%
#6\nobreak\relax
\dotfill\hbox to\@pnumwidth{\@tocpagenum{#7}}\par
\nobreak
\endgroup
\fi}
\makeatother

\newcommand{\R}{\mathbb{R}}

\newcommand{\Ha}{\mathcal{H}}
\newcommand{\F}{\mathcal{F}}

\newcommand{\beq}{\begin{equation}}
\newcommand{\eeq}{\end{equation}}
\newcommand{\dist}{{\rm dist}}
\newcommand{\vphi}{\varphi} 
\newcommand{\la}{\langle}
\newcommand{\ra}{\rangle}
\newcommand{\Div}{\operatorname{div}}
\newcommand{\pa}{\partial}
\newcommand{\medint}{-\kern -,375cm\int}
\newcommand{\medintinrigo}{-\kern -,315cm\int}

\newcommand{\dx}{\, {\rm d} x}

\newcommand{\de}{\, {\rm d}}

\allowdisplaybreaks

\begin{document}

\title[Flat Flow]{Flat flow solution to the mean curvature flow with volume constraint}

\author[Julin]{Vesa Julin}
\address{Vesa Julin,
Department of Mathematics and Statistics,
University of Jyv\"askyl\"a,
P.O. Box 35,
40014 Jyv\"askyl\"a,
Finland}
\email{\color{blue} vesa.julin@jyu.fi}


\date{\today}

\begin{abstract}
In this paper I will revisit the construction of a global weak solution to the volume preserving mean curvature flow via discrete minimizing movement scheme by Mugnai-Seis-Spadaro \cite{MSS}. This method is based on the gradient flow approach due to Almgren-Taylor-Wang \cite{ATW} and Luckhaus-Str\"urzenhecker \cite{LS} and my aim is to replace the volume penalization by implementing  the volume constraint directly in the discrete scheme, which from practical point of view is perhaps more natural. A technical novelty is the proof of the density estimate which is based on  the second variation condition of the energy. 
\end{abstract}

\maketitle
\setcounter{tocdepth}{2}

\section{Introduction}

A smooth family of set  $(E_t)_{t \geq 0}$ is said to evolve according to volume preserving mean curvature flow if the normal velocity $V_t$ is proportional to the mean curvature  $H_{E_t}$ as
\beq
\label{the flow}
V_t = -(H_{E_t} - \bar H_{E_t}) \qquad \text{on }\, \pa E_t, 
\eeq
where  $\bar H_{E_t} = \fint_{\pa E_t} H_{E_t} \, d \Ha^n$.  Such a geometric equation  has been
proposed in the physical literature to  model coarsening phenomena, where the system consisting on several subdomains  evolve such that it decreases the interfacial area while keeping  the total volume unchanged \cite{CRCT, MuSe}. From purely mathematical point of view the equation \eqref{the flow} can be seen as the $L^2$-gradient flow of the surface area under the volume constraint \cite{MuSe}. One has to be careful in this  interpration  as  the Riemannian distance  between two sets is in general degenerate \cite{MM}.  In order to overcome this one may use the idea due to Almgren-Taylor-Wang \cite{ATW} and Luckhaus-Str\"urzenhecker \cite{LS} and to view  \eqref{the flow} as the gradient flow of the surface area  with respect to a different, non-degerate, distance. Using the gradient flow structure, one may then construct a discrete-in-time approximation to the solution of \eqref{the flow} via the Euler implicit method, also known as the minimizing movements scheme. By letting the time step to zero, one then obtains a candidate for a weak solution of  \eqref{the flow}  called \emph{flat flow}, as the convergence is measured in terms of ''flat norm''.  This method is implemented to the volme preserving setting in \cite{MSS}.

 In  \cite{MSS} the authors observe that from technical point of view it is easier to replace the volume constraint of the problem with a volume penalization, as this simplifies certain regularity issues at the level of the discrete approximation. My aim here is to show that one may construct the flat flow solution to \eqref{the flow}  by implementing the volume constraint in the   minimizing movements scheme directly and thus avoid the volume penalization.    

Let me quickly recall the  discrete minimizing movements scheme for \eqref{the flow}. One defines a sequence of sets $(E_k^h)_{k}$, with fixed time step $h>0$, iteratively such that $E_0^h = E_0$, where $E_0$ is the given initial set, and $E_{k+1}^h$ is a minimizer of the functional
\[
P(E) + \frac{1}{h} \int_{E} \bar d_{E_k^h} \, dx \qquad \text{under the constraint } \, |E| = |E_k^h|. 
\]  
Here $P(E) $ denotes the perimeter (generalized surface area) of the set $E$ and  $\bar d_{F} $ is the signed distance function of  the set $F$ (see next section). One then defined an approximative flat flow solution to \eqref{the flow} $(E_t^h)_{ t\geq 0}$ from the previous sequence by $E_t^h = E_k^h$ for $t \in [kh,(k+1)h)$. Any cluster point of $(E_t^h)_{ t\geq 0}$ is then defined as flat flow solution to \eqref{the flow}.   The advantage is that such a solution is defined for all times and for rough initial data. The main result in the paper is the existence of a flat flow solution.

\begin{theorem}
\label{thm1}
Assume that $E_0 \subset \R^{n+1}$  is an open and bounded set with finite perimeter and let $(E_t^h)_{t \geq 0}$ be an approximative flat flow solution to \eqref{the flow} stating from $E_0$ (see Definition \ref{def:flatflow}). Then there exists a family of bounded sets  of finite perimeter $(E_t)_{t \geq 0}$ and a subsequence $h_k \to 0$ such that 
\[
\lim_{h_k \to 0} |E_t^{h_k} \Delta E_t| = 0 \qquad \text{for a.e. }  \, t \geq 0 
\]
and for every $0 <t <s$  it holds $|E_t| = |E_0|$, $P(E_t) \leq P(E_0)$ and 
\[
|E_t \Delta E_s| \leq C \sqrt{s-t},
\]
where $C$ depends on the dimension and on $E_0$. Moreover, if the initial set $E_0$ is $C^{1,1}$-regular, then any such limit flow $(E_t)_{t \geq 0}$ agrees with the unique classical solution of \eqref{the flow} as long as the latter exists.  
\end{theorem}

The above theorem thus provides the existence of a flat flow solution and guarantees that this notion is consistence with the classical solution when the initial set is regular enough. The disadvantage of the flat flow is that it is not clear if it provides a solution to the original equation \eqref{the flow} in any weak sense.    However,    the conditional result in the spirit of  Luckhaus-Str\"urzenhecker \cite{LS} holds also in this case. 
\begin{theorem}
\label{thm2}
Let  $(E_t^h)_{t \geq 0}$ be an approximative flat flow solution to \eqref{the flow} and let $(E_t^{h_k})_{t \geq 0}$ be the converging subsequence in Theorem \ref{thm1}. Assume further that it holds 
\[
\lim_{h_k \to 0} P(E_t^{h_k}) = P(E_t) \qquad \text{for a.e. }\,  t \geq 0. 
\]
Then for $n \leq 6$ the flat flow $(E_t)_{t \geq 0}$ is a distributional solution to \eqref{the flow} (see Definition \ref{def:distributional}).   
\end{theorem}

One may also try to view the equation \eqref{the flow} as a mean curvature flow with  forcing, where the forcing term depends on the flow itself. In this way one may try use different methods to construct a solution to the equation see e.g. \cite{BCC, BS} .     I also refer the recent work  \cite{Laux} for a weak-strong uniqueness result related to \eqref{the flow}.

As I already mentioned, the flat flow is defined for all times and one may study its asymptotical behavior. Indeed, by using the metods from  \cite{JuMoPoSpa, JN} one may deduce the convergence of the flow in low dimensions.

\begin{remark}
\label{rem:main}
Assume that $E_0 \subset \R^{n+1}$, with $n \leq 2$,   is as in Theorem \ref{thm1} and let $(E_t)_{t \geq 0}$ be a limit flat flow.  When $n=1$, the flow $E_t$ converges to  a union of disjoint balls exponentially fast and when $n=2$ the flow converges to a union of disjoint balls up to a possible translation of the components. 
\end{remark}

The main technical  challenge in proving  Theorem \ref{thm1} is to obtain the sharp density estimate for the discrete flow. This is also the main technical novelty of this paper. There are several techniques to deal with the volume constraint in variational problems, e.g. by using  the argument from \cite{Alm} (see also \cite[Lemma 17.21]{Ma}) or from \cite{GMT} (see also \cite{APP, Gru}). However, due to the presence of the dissipation term in the energy it is not obvious how to apply  these arguments  in order to obtain sharp density estimates in terms of the time step $h$.  I will use an argument which is based on the second variation condition of the energy to prove the density estimate in  Proposition \ref{prop:density}. After this the proof of Theorem \ref{thm1} follows exactly as in \cite{LS, MSS} and the consistency follows almost directly using the argument in \cite{JN2}. The proof also provides the dissipation inequality and therefore the results in \cite{JuMoPoSpa, JN} hold and one obtains the result stated in Remark \ref{rem:main}. Finally I would like to point out  that this article is not self-consistence as several arguments are well-known, in particular,  in Section 4.

\section{Preliminaries}

In this section I will briefly introduce the notation, the definition of the flat flow solution and recall some of its basic properties. 

Given a set $E \subset \R^{n+1}$   the distance function $\dist(\cdot, E) :  \R^{n+1}  \to [0,\infty)$ is defined, as usual, as $\dist(x, E)  := \inf_{y \in E} |x - y|$ and  denote the signed distance function by $\bar  d_E : \R^{n+1} \to \R$, 
\[
 \bar d_E(x) := \begin{cases} - \dist(x, \pa E)  , \,\,&\text{for }\, x \in E \\
 \dist(x, \pa E) , \,\, &\text{for }\, x \in \R^{n+1} \setminus E.  \end{cases} 
\]
Then clearly it holds $ \dist(\cdot , \pa E)  = |\bar d_E|$. I denote the ball with radius $r$ centered at $x$ by $B_r(x)$ and by $B_r$ if it is centered at the origin.

 For a measurable set $E\subset \R^{n+1}$ the perimeter in an open set $U \subset \R^{n+1}$ is defined by
\[
P(E, U) := \sup \Big{\{} \int_E \Div X \, dx : X \in C_0^1(U,\R^{n+1}) , \, \| X\|_{L^\infty} \leq 1 \Big{\}}
\]
and write $P(E) = P(E,\R^{n+1})$. If $P(E) < \infty$, then  $E$ is called  a set of finite perimeter. For an introduction to the topic I refer to  \cite{Ma}. The reduced boundary of a set of finite perimeter $E$ is denoted by  $\pa^* E$ and the generalized unit outer normal by $\nu_E$. Note that it holds $P(E, U)  = \Ha^n(\pa^* E \cap U)$ for open sets $U$. Recall also that if $E$ is regular enough, say with Lipschitz boundary, then $P(E) = \Ha^n(\pa E)$. For a given vector field $X \in C^1(\R^{n+1},\R^{n+1})$ and a set of finite  perimeter $E$  denote the tangential divergence on $\pa E^*$ as $\Div_\tau X  = \Div X - \la D X \nu_{E} , \nu_E \ra $. The distributional mean curvature  $H_E \in L^1(\pa^* E, \R)$ is defined via the divergence thereon such that for every test vector field  $X \in C_0^1(\R^{n+1},\R^{n+1})$ it holds 
\[
\int_{\pa^* E} \Div_\tau X  \,  \de  \Ha^n = \int_{\pa^* E} H_E \la X, \nu_E \ra  \,  \de  \Ha^n .
\]

I will consider a flat flow solution to \eqref{the flow} in the spirit of Almgren-Taylor-Wang \cite{ATW} and Luckhaus-St\"urzenhecker \cite{LS}. To this this aim for a fixed  $h \in (0,1)$ and a given (open) set $F \subset \R^{n+1}$ I define the functional 
\beq 
\label{def:functional}
\F_h(E,F) = P(E) + \frac{1}{h}\int_{E} \bar d_F \dx.
\eeq

The flat flow solution is defined analogously as in \cite{MSS}.

\begin{definition}
\label{def:flatflow}
Let $E_0 \subset \R^{n+1}$ be an open and bounded set of finite perimeter and fix $h \in (0,1)$. Define the sequence of sets $(E_k^h)_{k=0}^\infty$ iteratively as $E_0^h = E_0$ and $E_{k+1}^h$ is a minimizer of the problem
\[
\min \Big{\{} \F_h(E,E_{k}^h)   : |E| = |E_0|  \Big{\}}.
\]
Moreover, define an approximative flat flow $(E_t^h)_{t \geq 0}$ for \eqref{the flow} starting from $E_0$ as
\[
E_t^h  =E_k^h \qquad \text{for } \, t \in [kh,(k+1)h).
\]
\end{definition}

One has to be carefull in the definition of the functional \eqref{def:functional} if the set $F$ is merely a set of finite perimeter as its value depends on the choice of the representative  of $F$. One may overcome this by choosing a proper representative of the set $F$. In my case this is not necessary, as  the regularity theorem below implies that  one may in fact assume the sets $E_k^h $ to be open. The difference in the \cref{def:flatflow} to the scheme in \cite{MSS} is that here  the minimizing problem is under volume constraint. On one hand this makes the minimization  problem more natural, but on the other hand, it makes the quantitative density estimates more difficult to prove.

For a given open and bounded set $F \subset \R^{n+1}$ consider the minimization problem 
\beq 
\label{def:minprob}
\min \Big{\{} \F_h(E,F)  : |E| = |F|  \Big{\}},
\eeq
where $\F_h(\cdot,F) $ is defined in \eqref{def:functional}.  One may use an argument similar to \cite{Gru} or \cite[Lemma 17.21]{Ma} to get rid of  the volume constraint in \eqref{def:minprob} and  deduce that a minimizer of  \eqref{def:minprob}  is a minimizer also for 
\beq
\label{def:minprob-const}
\min \Big{\{} \F_h(E,F)  + \tilde \Lambda \big| |E| - |F| \big| \Big{\}},
\eeq
when $\tilde \Lambda$ is chosen large. Note that the constant  $\tilde \Lambda$ may have unoptimal dependence on $E$ and on $h$. However,  the property \eqref{def:minprob-const} is enough   deduce qualitative regularity properties since it implies that the minimizer inherits the regularity from the theory of the perimeter minimizers \cite{Ma}. One may also write the Euler-Lagrange equation and by standard calculations (see e.g. \cite{AFM}) we have the second variation condition. We state this in the following proposition. 
\begin{proposition}
\label{prop:regularity}
Let $F \subset \R^{n+1}$ be an open and bounded set, fix $h \in (0,1)$ and  let $E$  be  a minimizer of \eqref{def:minprob}. Then $E$ can be chosen to be open, which topological boundary is $C^{2,\alpha}$-regular up to a relatively closed singular set which Hausdorff dimension is at most $n-7$. In fact, the regular part is exactly the reduced boundary $\pa^* E$. 

The Euler-Lagrange equation
\beq
\label{eq:euler}
\frac{d_{F}}{h} = - H_E + \lambda,
\eeq  
where $\lambda \in \R$ is the Lagrange-multiplier, holds point wise on $\pa^* E$  and in a distributional sense on $\pa E$. The quadratic form associated with the second variation of the energy is non-negative, i.e., for all $\vphi \in H^1(\pa^* E)$ with $\int_{\pa^* E} \vphi \de \Ha^n = 0$ it holds
   \beq
\label{eq:2ndVar}
\int_{\pa^* E} |\nabla_\tau \vphi|^2 - |B_E|^2 \vphi^2 \, \de \Ha^n + \frac{1}{h}   \int_{\pa^* E} \la\nabla \bar d_F, \nu_E\ra  \vphi^2 \, \de \Ha^n \geq 0, 
\eeq  
where $B_E(x)$ denotes the second fundamental form  at $x \in \pa^* E$. 
\end{proposition}

\begin{proof}
Since the argument is standard, I will only give the outline. As I already mentioned,  the minimizer $E$ is also a minimizer of the problem \eqref{def:minprob-const} for some large constant $\tilde \Lambda$, which depend on $h$ and on $E$ itself. This implies that the set $E$ is a $\Lambda$-minimizer of the perimeter and thus the reduced boundary $\pa^* E$ is relatively open,  $C^{1,\alpha}$-regular hypersurface and the singular set $\pa E \setminus \pa^* E$ has dimension  at most $n-7$ \cite{Ma}. The $C^{2,\alpha}$-regularity then follows from the Euler-Lagrange equation and from standard Schauder-estimates for elliptic PDEs. 

One may   obtain the second variation condition \eqref{eq:2ndVar}  by using the argument from \cite{AFM}. Indeed, given a function  $\vphi \in C_0^1(\pa^* E)$ with $\int_{\pa^* E} \vphi \de \Ha^n = 0$, we may construct a family of diffeomorphisms $\Phi_t$ such that $\Phi_0 = id$, $|\Phi_t(E)| = |E|$ and $\frac{\pa}{\pa t}\big|_{t=0}\Phi_t(x)\cdot \nu_E = \vphi $. Then the inequality follows from the minimality of $E$ as
\[
\frac{\pa^2}{\pa t^2} \big|_{t=0} \F_h(\Phi_t(E),F) \geq 0 
\] 
and following the standard calculation of the second variation (see e.g. \cite{AFM}). Finally one obtains \eqref{eq:2ndVar} for all  $\vphi \in H^1(\pa^* E)$ by approximation argument and by the fact that the singular set has zero capacity. 
\end{proof}

\section{Density estimates} \label{sec:density}

This section  is the theoretical core of the paper. The aim is to prove the following density estimate. 
\begin{proposition}
\label{prop:density}
Let $F \subset \R^{n+1}$ be an open and bounded set of finite perimeter, fix $h \in (0,1)$ and  let $E$  be  a minimizer of \eqref{def:minprob}. Then  there is a  constant $c>0$, which   depends on the dimension $n$, $|F|$ and on  $P(F)$ such that for all $r \leq  \sqrt{h}$ and all $x \in \pa E$ it holds 
\[
\min \big{\{} |E \cap B_{r}(x)| , |B_{r}(x) \setminus E| \big{\}} \geq c r^{n+1}
\]
and for all $r \leq C_0 \sqrt{h}$, where $C_0 \geq1$, it holds 
\[
c r^{n} \leq P(E, B_r(x)) \leq C_1 r^n,  
\]
where $C_1$ depends also  on $C_0$. Moreover,  the following holds 
\[
\|H_E\|_{L^{\infty}(\pa^* E)} \leq \frac{1}{c \sqrt{h}} \quad \text{and} \quad \|\bar d_F\|_{L^{\infty}(\pa E)} \leq  c^{-1} \sqrt{h}.
\]
\end{proposition}

It is interesting that in \cite[Corollary 3.3]{MSS} the authors obtain similar result for their scheme for a constant which is independent of $P(F)$. 

I need several lemmas in order to prove  \cref{prop:density} and therefore I postpone its proof to the end of the section. Before proceeding to technical details, I state a useful consequence  of Proposition \ref{prop:density}. 
\begin{proposition}
\label{coro:minima}
Let $F, E \subset \R^{n+1}$ be as in Proposition \ref{prop:density}. Then there are constants $C\geq 1, c>0$ and $h_0 >0$, depending on the dimension, $|F|$ and $P(F)$ such that $E$ is $(\Lambda, r)$- minimizer of the perimeter for  $\Lambda = \frac{C}{\sqrt{h}}$,  $r = c \sqrt{h}$ and for $h < h_0$.  To be more precise, for   sets $G\subset \R^{n+1}$ with $E \Delta G \subset B_{c \sqrt{h}}(x_0)$  it holds 
\[
P(E) \leq P(G) +\frac{C}{\sqrt{h}} |E \Delta G|.
\] 
\end{proposition}

\begin{proof}
The argument is standard but I recall it for the reader's convenience. Let me first show that there is $x \in E$ and $\tilde c>0$ such that for $\rho = \tilde c \sqrt{h}$ it holds $B_\rho(x) \subset E $.  Fix $\rho$  and apply Besicovich covering theorem to find disjoint balls $\{B_\rho(x_i)\}_{i=1}^N$ such that $x_i \in E$ and 
\begin{equation} \label{prop:coro1}
\sum_{i = 1}^N|B_\rho(x_i)| = N |B_1| \rho^{n+1} \geq c |E|. 
\end{equation}
I claim that for some $i = 1,2, \dots, N$ it holds  $B_{\rho/2}(x_i)\subset E$. Indeed, if this is not the case then Proposition \ref{prop:density} implies
\[
P(E, B_\rho(x_i)) \geq c \rho^n 
\]
for all $i$.  Since the balls are disjoint, one has by the above and by \eqref{prop:coro1}
\[
P(E) \geq \sum_{i = 1}^NP(E, B_\rho(x_i)) \geq c N \, \rho^n \geq c \frac{|E|}{\rho} \geq \frac{c}{\sqrt{h}}.
\]
This is a contradiction when $h$ is small enough.

Fix $x_0 \in \pa E$ and $G$ as in the claim.  Note that in general the set $G$ does not have the same measure as $E$ and one needs to modify it to $\tilde G$ with $|\tilde G|= |E|$ e.g. by using the argument from \cite{Gru} as follows. Assume that $|G|< |E|$ (the case $|G|> |E|$ follows from similar argument). Since  $B_{\rho}(x) \subset E$, then by decreasing $\rho$ and $r$ if needed, it holds $B_{\rho}(x) \subset G$. By continuity there is $z \in \R^{n+1}$  such that $|z - x_0| \geq 2 \rho$  and $|G \cup B_{\rho}(z)| = |E|$. Define $\tilde G = G \cup B_{\rho}(z)$. Then by the minimality of $E$ and Proposition \ref{prop:density}  it holds
\[
P(E) \leq P(\tilde G) + \frac{C}{\sqrt{h}}|\tilde G \Delta E |. 
\]
Arguing as in  \cite{Gru}   one then  deduces
\[
\begin{split}
P(\tilde G)  - P(G) &\leq \Ha^n(\pa B_{\rho}(z) \setminus G) - \Ha^n(\pa G \cap B_{\rho}(z)) \\
&\leq  \frac{C}{\rho} |B_{\rho}(z) \setminus G | \leq \frac{C}{\sqrt{h}}|\tilde G \Delta E|
\end{split}
\]
and the claim follows as $|\tilde G \Delta E| \leq 2 |G \Delta E|$ . 
\end{proof}

The first technical result which I need  is the classical density estimate which can be found e.g. in \cite{Sim}. 
\begin{lemma}
\label{lem:density1}
Assume $E \subset \R^{n+1}$ is a set of finite perimeter with distributional mean curvature $H_E$ which satisfies $\| H_E\|_{L^\infty(B_{2R}(x_0))} \leq \Lambda$. Then for all $x \in B_R(x_0)$ and  $r \leq \min \{ R, \Lambda^{-1}\}$ it holds
\[
P(E, B_r(x)) \geq c_n r^n,
\]  
for a dimensional constant $c_n>0$. 
\end{lemma}

For a minimizer of  \eqref{def:minprob} it holds the inverse of the isoperimetric inequality. 
\begin{lemma}
\label{lem:density2}
Let $F$ and $ E$ be as in \cref{prop:density}. Then for all $x \in \pa E$ and $r \leq C_0\sqrt{h}$ it holds 
\[
P(E,B_r(x)) \leq \frac{C}{r} \min \big{\{} |E \cap B_{2r}(x)| , |B_{2r}(x) \setminus E| \big{\}},
\]
for a constant which depends on the dimension and on $C_0>0$. In, particular it holds  $P(E,B_r(x))  \leq C  r^n$.
\end{lemma}

\begin{proof}
Fix $h \in (0,1)$, $x$ and $r>0$ as in the claim and without loss of generality assume that $x = 0$. One may also consider only the case $ |E \cap B_{2r}| \leq |B_{2r}\setminus E|$ as the other case is similar. In particular, then it holds  $|E \cap B_{2r}|  \leq \frac12 |B_{2r}|$. 

Since
\[
\int_0^{2r} \Ha^n(\pa B_\rho \cap E) = |E \cap B_{2r}| ,
\]
there is  $ \rho \in (r, 2r)$ such that 
\beq \label{eq:density2-1}
\Ha^n(\pa B_\rho \cap E) \leq  C_n \frac{|E \cap B_{2r}| }{r} \quad \text{and} \quad |B_{\rho}| \geq \frac23 |B_{2r}|. 
\eeq
Consider first the set $E_1 = E \setminus \bar B_\rho $. In order to have a competing set with the volume of $E$, define $\tilde \rho < \rho$ to be a radius such that $|B_{\tilde \rho}| = |E \cap B_\rho|$ and define $E_2 = E_1 \cup B_{\tilde \rho}$. Then it holds by construction that $|E_2| = |E|$.   By the minimality of $E$ we have 
\[
P(E) + \frac{1}{h}\int_{E} \bar d_F \dx \leq P(E_2) + \frac{1}{h}\int_{E_2} \bar d_F \dx.
\]
Estimate the perimeter of $E_2$ using \eqref{eq:density2-1} as
\[
\begin{split}
P(E_2) &\leq  P(E, \R^{n+1}\setminus  B_\rho) + \Ha^n(\pa B_\rho \cap E) + \Ha^n(\pa B_{\tilde \rho}) \\
&\leq  P(E, \R^{n+1}\setminus  B_\rho) + C_n \frac{|E \cap B_{2r}| }{r}  . 
\end{split}
\]
Use then   $E \Delta E_2 \subset B_{2r}$, $|E| = |E_2|$  and the fact that the signed distance function is $1$-Lipschitz to estimate 
\[
\big|\int_{E } \bar d_F \dx - \int_{E_2} \bar d_F\big| \leq  4r |E \cap B_\rho| \leq  4r |E \cap B_{2r}|.  
\]
Therefore one obtains by combining the three above inequalities and $r \leq  C_0 \sqrt{h}$
\[
P(E, B_r) \leq P(E, B_\rho) \leq C_n \frac{|E \cap B_{2r}| }{r}  + \frac{4r}{h} |E \cap B_{2r}| \leq C  \frac{|E \cap B_{2r}| }{r}  . 
\]
\end{proof}

 By Lemma \ref{lem:density1} and Lemma \ref{lem:density2} it is clear  that for \cref{prop:density} it is crucial  to prove the curvature estimate $\|H_E\|_{L^{\infty}} \leq \frac{C}{\sqrt{h}}$. The next lemma is a step towards this.

\begin{lemma}
\label{lem:curvaturebound1}
Let $F,  E$ and $h$ be as in \cref{prop:density}. Then it holds 
\[
\|H_E\|_{L^{2}(\pa^* E)} \leq \frac{C_1}{\sqrt{h}},
\]
where the constant $C_1$ depends on the dimension and on $|F|$ and $P(F)$.  
\end{lemma}

\begin{proof}
The proof relies on the second variation inequality in Proposition \ref{prop:regularity}. I would like to point out that in the case of the mean curvature flow, when there is no volume constraint, the proof is considerable easier as one could choose constant function in   \eqref{eq:2ndVar}. In the volume preserving case I will choose a cut-off function for a test function.   

To this aim  use  first \cite[Proposition 2.3]{JN} (see also \cite[Lemma 2.1]{MoPoSpa}) to find a point $x_0 \in \R^{n+1}$ and a radius $ r \in (c,1)$, where $c =  c(n,|F|, P(F))$, such that 
\[
|E \cap B_r(x_0)| = \frac12 |B_r|.
\]
 Note that the minimality of $E$ yields $P(E) \leq P(F)\leq C$. Moreover, by  the isoperimetric inequality it holds  $P(E) \geq c_n|E|^{\frac{n}{n+1}} = c_n|F|^{\frac{n}{n+1}} \geq c$. These  estimates are used repeatedly from now on without mentioning. Without loss of generality assume that $x_0 = 0$.  Choose $\rho  < r $ such that $|B_r \setminus B_\rho| = \frac14|B_r|$. Note that then $r - \rho \geq c_n >0$ and 
\beq \label{eq:curvaturebound1-0}
\frac34 |B_\rho| \geq |E \cap B_\rho| \geq \frac14 |B_\rho|.
\eeq
Define  first a cut-off function $\zeta \in C_0^1(\R^{n+1})$ such that $0 \leq \zeta \leq 1$, $\zeta = 1$ in $B_\rho$, $\zeta = 0$ outside $B_r$ and $|\nabla \zeta| \leq C_n$. Choose then $\vphi = \zeta - \bar \zeta$, where $\bar \zeta = \fint_{\pa^* E} \zeta \de \Ha^n$, as a test function in \eqref{eq:2ndVar}, use $|\la\nabla \bar d_F, \nu_E \ra| \leq 1$  and $|\vphi|\leq 1$, and obtain
\beq \label{eq:curvaturebound1-1}
\int_{\pa^* E}|B_E|^2 (\zeta - \bar \zeta)^2 \, \de \Ha^n \leq  \int_{\pa^* E} |\nabla_\tau \zeta |^2 \, \de \Ha^n + \frac{P(E)}{h}  .
\eeq
 Since $H_E = \text{Trace}(B_E)$,  it holds point wise on $\pa^* E$
\beq \label{eq:B-H}
|B_E|^2 \geq \frac{H_E^2}{n} . 
\eeq
Recall that  $0 \leq \zeta \leq 1$. Moreover, by the isoperimetric inequality and by \eqref{eq:curvaturebound1-0} it holds $P(E, B_{\rho}) \geq  c_n |E \cap B_\rho|^{\frac{n}{n+1}} \geq c $. Therefore $ \bar \zeta \geq c$ for $c = c(n,|F|, P(F))$. In particular, it holds $|\zeta(x) - \bar \zeta|  \geq c$ for $x \in \pa^* E \setminus B_r$. Hence, we have  by \eqref{eq:curvaturebound1-1}
\[
\int_{\pa^* E \setminus B_r}|H_E|^2 \, \de \Ha^n  \leq  \frac{C}{h}P(F) .
\]

We repeat the same argument by defining a cut-off function $\zeta \in C^1(\R^{n+1})$ as $\zeta = 0$ in $B_r$, $\zeta = 1$ outside $B_R$ and $|\nabla \zeta| \leq C_n$, where $R>r$ is such that $|B_R \setminus B_r| = \frac14|B_r|$. Using $\vphi = \zeta - \bar \zeta$  in  \eqref{eq:2ndVar} and arguing as above yields
\[
\int_{\pa^* E \cap  B_r}|H_E|^2  \, \de \Ha^n   \leq \frac{C}{h}P(F) 
\]
and the claim follows. 

\end{proof}

The last lemma I need is a bound on the Lagrange multiplier in the Euler-Lagrange equation \eqref{eq:euler}. 
\begin{lemma}
\label{lem:lagrangebound}
Let $F,  E$ and $h$ be as in \cref{prop:density}.  Then for the Lagrange multiplier in  \eqref{eq:euler}, i.e., 
\[
\frac{\bar d_{F}}{h} = - H_E + \lambda \quad \text{on } \, \pa^* E
\]
it holds 
\[
| \lambda | \leq  \frac{C_2}{\sqrt{h}},
\]
where the constant $C_2$ depends on the dimension, on $|F|$ and  on $P(F)$.  
\end{lemma}

\begin{proof}
Let $\Lambda \geq 0$ be such that $|\lambda | =   \frac{\Lambda}{\sqrt{h}}$.  Below all the constant depend on $n, |F|$ and $P(F)$. I only treat the case when $\lambda$ is positive as in the negative case the proof is the similar. Define the set 
\[
\Sigma = \{ x \in \pa^* E : |H_E(x)| < \frac{\hat C}{\sqrt{h}} \}.
\]  
 I claim that we may choose  $\hat C>2$ such that it  depends on $n, |F|, P(F)$ and on $C_1$ from \cref{lem:curvaturebound1} and it holds
\beq \label{eq:lagrangebound-1}
\Ha^n(\Sigma) \geq \frac{P(E)}{2}. 
\eeq
Indeed, by  the Euler-Lagrange equation \eqref{eq:euler} , by Lemma \ref{lem:curvaturebound1}  and by $\hat C>2$ it holds
\[
\frac{\hat C^2}{h}   \, \Ha^n(\pa^* E \setminus \Sigma) \leq \int_{\pa^* E \setminus \Sigma} H_E^2\, \de \Ha^n \leq    \int_{\pa^* E} H_E^2\, \de \Ha^n \leq \frac{C_1}{h}.
\]
By choosing $\hat C$ large enough one then obtains $\Ha^n(\pa^* E \setminus \Sigma) < P(E)/2$ and  \eqref{eq:lagrangebound-1} follows. 

By Besikovich covering theorem one finds disjoint balls of radius $\sqrt{h}$,  denote them $\{B_{\sqrt{h}}(x_i)\}_{i=1}^N$, with $x_i \in \Sigma$ such that 
\beq \label{eq:lagrangebound-2}
\sum_{i=1}^N P(E,  B_{\sqrt{h}}(x_i)) \geq c_n \Ha^n(\Sigma) \geq c P(E). 
\eeq
By the Euler-Lagrange equation \eqref{eq:euler} and by the definition of the set $\Sigma$  it holds for all $x \in \pa^* E \cap  B_{\sqrt{h}}(x_i)$ with $x_i \in \Sigma$ that 
\beq \label{eq:lagrangebound-3}
\begin{split}
|H_E(x)| &\leq \big|\frac{\bar d_F(x)}{h}  \lambda  \big| \leq \frac{|\bar d_F(x)- \bar d_F(x_i)|}{h} +  \big| \frac{\bar d_F(x_i)}{h} - \lambda  \big|\\
&\leq  \frac{1}{\sqrt{h}} +  |H_E(x_i)| \leq  \frac{2\hat C}{\sqrt{h}} .
\end{split}
\eeq
Therefore by Lemma \ref{lem:density1} it holds 
\[
P(E, B_{\sqrt{h}/2}(x_i)) \geq  P(E, B_{\sqrt{h}/2 \hat C}(x_i)) \geq c h^{n/2} .
\]
On the other hand applying Lemma \ref{lem:density2} first  with  $r =\sqrt{h}/2 $ yields
\[
|E \cup B_{\sqrt{h}}(x_i)| \geq c \sqrt{h}   P(E,  B_{\sqrt{h}/2}(x_i))   
\]
and then with $r =\sqrt{h}$ yields $h^{\frac{n+1}{2}} \geq c P(E,  B_{\sqrt{h}}(x_i))$. In conclusion it holds 
\beq \label{eq:lagrangebound-4}
  |E \cup B_{\sqrt{h}}(x_i)| \geq  c \sqrt{h}   P(E,  B_{\sqrt{h}}(x_i))  
\eeq
 for all balls in the cover. 

The minimality of $E$ implies 
\beq \label{eq:lagrangebound-5}
 \frac{1}{h} \int_{E \setminus F} \bar d_F\, \dx \leq P(E) + \frac{1}{h} \int_{E \Delta F} |\bar d_F|\, \dx \leq P(F) .
\eeq
Note that by \eqref{eq:lagrangebound-3}, by $\lambda =   \frac{\Lambda}{\sqrt{h}}$ and by  the Euler-Lagrange equation \eqref{eq:euler} we have for all $x \in B_{\sqrt{h}}(x_i)$ that 
\[
\bar d_F(x) \geq \lambda h  -|H(x)| h \geq (\Lambda- 2\hat C) \sqrt{h}.
\]
Therefore either $\Lambda \leq 4\hat C$, in which case the claim follows trivially, or 
\[
\bar d_F(x) \geq \frac{\Lambda}{2}\sqrt{h}.
\]
I assume the latter and show that even in this case $\Lambda$ is bounded. Indeed, by the above discussion  the balls $B_{\sqrt{h}}(x_i)$  are in the exterior of $F$. Therefore we estimate by \eqref{eq:lagrangebound-2} and \eqref{eq:lagrangebound-4} that 
\[
\begin{split}
\frac{1}{h} \int_{E \setminus F} \bar d_F\, \dx &\geq \frac{1}{h} \sum_{i=1}^N \int_{E \cap B_{\sqrt{h}}(x_i)} \bar d_F\, \dx\\
&\geq  \frac{1}{h} \sum_{i=1}^N    \left(  \frac{\Lambda}{2}\sqrt{h}   |E \cap B_{\sqrt{h}}(x_i)| \right) \\
&\geq  c \Lambda    \sum_{i=1}^N   P(E,  B_{\sqrt{h}}(x_i)) \\
&\geq c \Lambda    \Ha^n(\Sigma) \geq c \Lambda P(E).
\end{split}
\]
Since $ P(E) \geq c_n |E|^{\frac{n}{n+1}} = c_n |F|^{\frac{n}{n+1}}$, the above and \eqref{eq:lagrangebound-5} gives  a bound for   $\Lambda$  and the claim follows.  
\end{proof}

Here is the proof of the density estimate. 

\begin{proof}[\textbf{Proof of \cref{prop:density}}]
By Lemma \ref{lem:density1}, Lemma \ref{lem:density2},  Lemma \ref{lem:lagrangebound} and by the Euler-Lagrange equation \eqref{eq:euler} it is enough to prove 
\beq \label{eq:curvbound-1}
\|\bar d_F\|_{L^{\infty}(\pa^* E)} \leq  C \sqrt{h}.
\eeq
Argue by contradiction and assume that  there is $x_0 \in \pa^* E$ such that 
\[
|\bar d_F(x_0)| =  \Lambda  \sqrt{h} 
\] 
for large $\Lambda>>1$.   Without loss of generality  assume that $x_0 = 0$ and  consider only the case $\bar d_F(x_0) >0$ as  the other case  is similar. 

By  the Euler-Lagrange equation \eqref{eq:euler}, by Lemma \ref{lem:density1} and Lemma \ref{lem:lagrangebound} it holds for $r_0 = \frac{ \sqrt{h}}{2\Lambda}  $ that 
\beq \label{eq:curvbound-10}
P(E, B_{r_0}) \geq c r_0^n. 
\eeq
Define  radii $r_k = k\sqrt{h} +r_0$ for $k =0,1,\dots, n+1$. For every  $k = 0,1,\dots, n+1 $ choose a cut-off function $\zeta_k \in C_0^1(\R^{n+1})$ such that $0 \leq \zeta \leq 1$, $\zeta_k = 1$ in $B_{r_k}$, $\zeta_{k+1} = 0$ in   $\R^{n+1} \setminus B_{r_{k+1}}$ and $|\nabla \zeta_k| \leq \frac{2}{\sqrt{h}}$. Choose $\vphi = \zeta_k - \bar \zeta_k$, where $\bar \zeta_k = \fint_{\pa^* E} \zeta_k \de \Ha^n$, as a test function in the second variation condition \eqref{eq:2ndVar}, use $|\la\nabla \bar d_F, \nu_E \ra| \leq 1$ and \eqref{eq:B-H},    and obtain
\beq \label{eq:curvbound-2}
\frac{1}{n}\int_{\pa^* E} |H_E|^2 (\zeta_k - \bar \zeta_k)^2 \, \de \Ha^n\leq    \frac{1}{h} \int_{\pa^* E} (\zeta_k - \bar \zeta_k)^2 \, \de \Ha^n +  \int_{\pa^* E} |\nabla_\tau \zeta_k|^2 \, \de \Ha^n . 
\eeq
Since $\zeta_k = 0$ outside $B_{r_{k+1}}$,  one may estimate
\[
\int_{\pa^* E} |\nabla_\tau \zeta_k|^2 \, \de \Ha^n \leq \frac{4}{h} P(E, B_{r_{k+1}}).
\]
Moreover,  since $0 \leq \zeta_k \leq 1$ it holds 
\[
 \int_{\pa^* E} (\zeta_k - \bar \zeta_k)^2 \, \de \Ha^n  = \int_{\pa^* E} \zeta_k^2 - \bar \zeta_k^2 \, \de \Ha^n \leq \int_{\pa^* E} \zeta_k^2 \, \de \Ha^n \leq P(E, B_{r_{k+1}}).
\]
When $\Lambda$ is large enough, then for all $x \in \pa^* E \cap B_{r_{k+1}}$ and all $k \leq n+1$ it holds $\bar d_F(x) \geq \frac{\Lambda}{2}  \sqrt{h}$. Then one may deduce from   \eqref{eq:lagrangebound-1}  that  $P(E,  B_{r_{k+1}}) \leq \frac12 P(E)$ when $\Lambda$ is large enough.  This yields 
\[
0 \leq  \bar \zeta_k \leq \frac{1}{2}.
\]
Also by the Euler-Lagrange equation \eqref{eq:euler}, by  $\bar d_F(x) \geq \frac{\Lambda}{2}  \sqrt{h}$, and by Lemma \ref{lem:lagrangebound} it holds  $|H_E| \geq \frac{\Lambda}{4 \sqrt{h}} $ on $\pa^* E \cap B_{r_{k}}$, when $\Lambda$ is large. Therefore it holds
\[
\begin{split}
\frac{1}{n}\int_{\pa^* E} |H_E|^2 (\zeta_k - \bar \zeta_k)^2 \, \de \Ha^n &\geq \frac{1}{n}\int_{\pa^* E\cap B_{r_{k}}} |H_E|^2 (1 - \bar \zeta_k)^2 \, \de \Ha^n \\
&\geq  c_n \frac{\Lambda^2}{h} P(E, B_{r_{k}}).
\end{split}
\] 
Combining the three above estimates with \eqref{eq:curvbound-2} yields
\[
c_n \frac{\Lambda^2}{h}   P(E,  B_{r_{k}}) \leq  \frac{5}{h} P(E, B_{r_{k+1}}). 
\]
For $\Lambda$ large enough this implies
\beq \label{eq:curvbound-3}
\Lambda   P(E,  B_{r_{k}}) \leq  P(E,  B_{r_{k+1}}). 
\eeq

Use  \eqref{eq:curvbound-3} $(n+1)$-times  from $k =0$ to $k = n$,  use then \eqref{eq:curvbound-10} and recall  that $r_0 = \frac{\sqrt{h}}{2\Lambda}$ and obtain finally that  
\beq \label{eq:curvbound-4}
 P(E, B_{r_{n+1}}) \geq \Lambda^{n+1}  P(E,  B_{r_{0}}) \geq c \Lambda^{n+1} r_0^n  = c \Lambda^{n+1} \left( \frac{\sqrt{h}}{2\Lambda}\right)^n = c \Lambda h^{n/2}.
\eeq
But now since $r_{n+1} = (n+1)\sqrt{h} + r_0 \leq 2(n+1)\sqrt{h} $, one obtains from  Lemma \ref{lem:density2} with  $r = 2(n+1)\sqrt{h}$ that 
\[
 P(E, B_{r_{n+1}}) \leq   P(E, B_{r})  \leq C  h^{n/2},
\]
which contradicts \eqref{eq:curvbound-4} when $\Lambda$ is large.  
\end{proof}

\section{Existence of the flat flow}

Now that the density estimates are proven the proof of Theorem \ref{thm1} follows from the arguments from  \cite{LS, MSS} without  major changes. 
In this section I  consider the approximative flat flow $(E_t^h)_{t \geq 0}$ and the associated sequence   $(E_k^h)_{k \geq 0}$ as in the Definition \ref{def:flatflow} starting from an open and bounded set of finite perimeter $E_0$. The proof for the  following ''interpolation'' result can be found in \cite[Lemma 1.5]{LS}. 
\begin{lemma}
\label{lem:interpol}
Let  $(E_t^h)_{t \geq 0}$ be an  approximative flat flow starting from $E_0$ and  fix $h \in (0,1)$ and $t >h$.  Then for all $l \leq \sqrt{h}$ it holds
\[
|E_t^h \Delta E_{t-h}^h| \leq C \left( l \, P(E_{t-h}^h) + \frac{1}{l} \int_{E_{t}^h \Delta E_{t-h}^h} | \bar d_{E_{t-h}^h} | \, dx  \right)
\]
\end{lemma}   

By the regularity result stated in Proposition \ref{prop:regularity}  the Euler-Lagrange equation
\begin{equation}\label{eq:euler-app}
\frac{\bar d_{E_{t-h}^h}}{h} = - H_{E_t^h} + \lambda_{t,h}
\end{equation}
holds point wise on $\pa^* E_t^h$ and in a distributional sense on $\pa E_t^h$. Here $\lambda_{t,h}$ is the Lagrange multiplier. Using the minimality of $E_{t}^h$ against the previous set $E_{t-h}^h$ one obtains the important inequality 
\begin{equation}\label{eq:dissipation}
P(E_{t}^h) + \frac{1}{h}  \int_{E_{t}^h \Delta E_{t-h}^h} | \bar d_{E_{t-h}^h} | \, dx  \leq P(E_{t-h}^h). 
\end{equation}

Using \eqref{eq:dissipation}  and  the argument \cite[Lemma 2.1]{LS} (see also \cite[Lemma 3.6]{MSS}) one  obtains the  following dissipation inequality.
\begin{lemma}
\label{lem:dissipation}
Let  $(E_t^h)_{t \geq 0}$ be an  approximative flat flow starting from $E_0$ and  fix $h \in (0,1)$. Then for all $T_2 >T_1 \geq h$ it   holds
\[
\int_{T_1}^{T_2} \| H_{E_t^h} -\lambda_{t,h} \|_{L^2(\pa^* E_t^h)}^2 \, dt \leq C( P(E_{T_1-h}) - P(E_{T_2})).  
\]
Moreover, it holds 
\[
\int_{T_1}^{T_2} (\|H_{E_t^h}\|_{L^2(\pa^* E_t^h)}^2  + \lambda_{t,h}^2)  \, dt \leq C(1+ T_2 -T_1). 
\]
The constant depends on the dimension, $|E_0|$ and $P(E_0)$. 
\end{lemma}   

\begin{proof}
 I will only  sketch the proof. Let $(E_k^h)$ be the sequence of sets associated with $(E_t^h)_{t \geq 0}$.  For $l \in \mathbb{Z}$ with $2^l \leq  2C h^{-\frac12}$ set 
\[
K(l)  =\{ x \in \R^{n+1}: 2^l h < |\bar d_{E_{t-h}^h}| \leq 2^{l+1} h \}.
\]  
Here $C$ is such that $ |\bar d_{E_{t-h}^h}| \leq C h $ on $\pa E_{t}^h$. Proposition \ref{prop:density} yields that for every $x \in \pa E_{t}^h$
\[
|E_{t}^h \cap B_{2^l h}(x)| \geq c (2^l h)^{n+1} \quad \text{and} \quad \Ha^n(\pa E_{t}^h \cap B_{2^l h}(x)) \leq C (2^l h)^{n}.
\]
Therefore for all $x \in \pa E_{t}^h \cap K(l)$ it holds 
\[
\begin{split}
&\int_{B_{2^l h}(x) \cap  E_{t}^h \Delta E_{t-h}^h} |\bar d_{E_{t-h}^h}|\, dx \geq c  (2^l h)^{n+2} \qquad \text{and}\\
&\int_{B_{2^l h}(x) \cap \pa E_{t}^h} \bar d_{E_{t-h}^h}^2 \, d \Ha^n \leq C  (2^l h)^{n+2}.
\end{split}
\]
Combing these two yields 
\[
\int_{B_{2^l h}(x) \cap \pa E_{t}^h} \bar d_{E_{t-h}^h}^2 \, d \Ha^n \leq C \int_{B_{2^l h}(x) \cap  E_{t}^h \Delta E_{t-h}^h} |\bar d_{E_{t-h}^h}|\, dx.
\]
By applying Besicovitch covering theorem and summing over $l \in \mathbb{Z}$  (see \cite[Lemma 3.6]{MSS} for details) yields
\[
\int_{ \pa E_{t}^h} \bar d_{E_{t-h}^h}^2 \, d \Ha^n \leq \int_{  E_{t}^h \Delta E_{t-h}^h} |\bar d_{E_{t-h}^h}|\, dx
\]
which by combining with \eqref{eq:euler-app} and \eqref{eq:dissipation} implies 
\[
h \int_{ \pa E_{t}^h} (H_{E_t^h} -\lambda_{t,h} )^2 \, d \Ha^n \leq C (P(E_{t-h}^h) -P(E_{t}^h) ).
\] 
The first inequality then follows by iterating the above. 

By \cite[Lemma 2.4]{JN} it holds 
\[
|\lambda_{t,h} | \leq C(1 + \|H_{E_t^h} -\lambda_{t,h}\|_{L^1(\pa^* E_t^h)})
\]
for a constant that depends on the dimension and on $|E_0|$ and $P(E_0)$. Note that  then 
\[
\|H_{E_t^h}\|_{L^2}^2 + \lambda_{t,h}^2 \leq   C(1 + \|H_{E_t^h} -\lambda_{t,h}\|_{L^2}^2). 
\]
Therefore by the first inequality one obtains 
\[
\begin{split}
 \int_{T_1}^{T_2} (\|H_{E_t^h}\|_{L^2}^2 + \lambda_{t,h}^2) \, dt \leq C \int_{T_1}^{T_2} (1 + \|H_{E_t^h} -\lambda_{t,h}\|_{L^2}^2) \, dt \leq C(1+ T_2 -T_1). 
\end{split}
\]
\end{proof}

The third lemma we need is a quantitative bound on the diameter of the sets $(E_t^h)$, which is essentially the same as \cite[Lemma 3.8]{MSS}. 
\begin{lemma}
\label{lem:bounded}
Let  $(E_t^h)_{t \geq 0}$ be an  approximative flat flow starting from $E_0$ for  $h \in (0,1)$. Then for all $T>0$ there is $R_T$, which depends on $T$, on the dimension and on the diameter of the initial set $E_0$,  such that $E_t^h \subset B_{R_T}$ for all $t \leq T$.
\end{lemma}   
\begin{proof}
As in  \cite[Lemma 3.8]{MSS} define $r_t$ for all $t \leq T$ as
\[
r_t := \inf \{ r>0 : E_t^h \subset B_r\}. 
\]  
Arguing as in  \cite[Lemma 3.8]{MSS} one deduces that at the point $y \in \pa B_{r_t}\cap \pa E_t^h$ it holds $H_{E_t^h}(y)\geq 0$ and therefore by \eqref{eq:euler-app} 
\[
r_t \leq r_{t-h} + h|\lambda_{t,h}|.
\]
Iterating this and using Lemma \ref{lem:dissipation} yields 
\[
R_T - R_0 \leq \int_{0}^T |\lambda_{t,h}| \, dt \leq \int_{0}^T (1+  \lambda_{t,h}^2) \, dt \leq C(1+T).
\]

\end{proof}

\begin{proof}[\textbf{Proof of Theorem \ref{thm1}}]
Let  $(E_t^h)_{t \geq 0}$ be an  approximative flat flow starting from $E_0$ for  $h \in (0,1)$ and fix $T \geq 1$. Then by \eqref{eq:dissipation} it holds $P(E_t^h) \leq P(E_0)$ and by Lemma \ref{lem:bounded} it holds $E_t^h \subset B_{R_T}$ for all $t \leq T$.  I claim that for $ 0<t < s $ with $s-t \geq h$  it holds 
\begin{equation}
\label{eq:holder-in-time}
|E_t^h \Delta E_s^h| \leq C \sqrt{t-s}.
\end{equation}
Once \eqref{eq:holder-in-time} is obtained, then the convergence of a subsequence $E_t^{h_k} \to E_t$ in measure follows as in \cite{LS, MSS}. 

Let $j, k$ be such that $s \in [jh, (j+1)h)$ and $t \in [(j+k)h, (j+k+1)h)$. Then by applying Lemma \ref{lem:interpol} for $l = \frac{h}{\sqrt{s-t}}$ and by \eqref{eq:dissipation} one obtains 
\[
\begin{split}
|E_t^h \Delta E_s^h|  &\leq \sum_{m=j}^{j+k} |E_{mh}^h \Delta E_{(m+1)h}^h| \\
&\leq C \sum_{m=j}^{j+k}\left( \frac{h}{\sqrt{s-t}} P(E_{mh}^h) + \frac{\sqrt{s-t}}{h}  \int_{E_{(m+1)h}^h \Delta E_{mh}^h} | \bar d_{E_{mh}^h} | \, dx \right)\\
&\leq  C \sum_{m=j}^{j+k}\frac{h}{\sqrt{s-t}} P(E_0) + C \sqrt{s-t} \sum_{m=j}^{j+k} (P(E_{mh}^h) - P(E_{(m+1)h}^h))\\
&\leq  C\frac{kh}{\sqrt{s-t}} P(E_0) +  C \sqrt{s-t} P(E_0). 
\end{split}
\]
Since $kh \leq 2(s-t)$, one obtains \eqref{eq:holder-in-time}.

The proof of the consistency principle for $C^{1,1}$-regular initial sets follows using the  arguments in \cite{JN2}. The volume penalization is used only in  \cite[Lemma 3.2]{JN2}, but one may overcome this by using the lemma below. 
\end{proof}

\begin{lemma}
\label{lem:tekninen1}
Let $F \subset \R^{n+1}$ be an open and bounded set which satisfies  interior and exterior ball condition with radius $r_0>0$  and  let $E$  be  a minimizer of \eqref{def:minprob}. There are $\rho_0$ and $h_0$ with the property that if $G$ is a set of finite perimeter such that 
\[
G \Delta E \subset B_{\rho}(x) \cap \mathcal{N}_{C_0h}(\pa F), 
\]
for $\rho \leq \rho_0$ and $h \leq h_0$ where $\mathcal{N}_{C_0h}(\pa F) = \{ x : \dist(x, \pa F) <  C_0h \}$, then it holds 
\[
P(E) \leq P(G) + C\rho^{n+1}. 
\]
Above the constant depends on the dimension, on $r_0, C_0, |F|$ and $P(F)$. 
\end{lemma}

\begin{proof}
By approximation one may assume $G$ to be smooth. Since $F$ satisfies interior and exterior ball condition, then by \cite[Lemma 3.1]{JN2} it holds
\begin{equation} \label{eq:tekninen-1}
\max_{x \in E \Delta F} \dist(x, \pa F) \leq C h 
\end{equation}
when $h \leq h_0$.  As in the proof of Proposition \ref{coro:minima} the set  $G$ does not have the same measure as $E$ and modifies it to $\tilde G$ with $|\tilde G|= |E|$.  Assume again that $|G|< |E|$. Since $F$ satisfies  interior ball condition with radius $r_0$, there is $y \in G$ such that $B_{r_0/2}(y) \subset G$. By continuity there is $z \in \R^{n+1}$  such that $|z - x| \geq 2 \rho_0$  and $|G \cup B_{r_0/2}(z)| = |E|$ when $\rho_0$ is small. Define $\tilde G = G \cup B_{r_0/2}(z)$. Then by the minimality of $E$ and by \eqref{eq:tekninen-1} and by the assumption  $G \Delta E \subset  B_{\rho}(x) \cap \mathcal{N}_{C_0h}(\pa F)$ it holds
\[
P(E) \leq P(\tilde G) + C|\tilde G \Delta E | \leq P(\tilde G) + C\rho^{n+1}. 
\]
Arguing as in  \cite{Gru}   one then deduces
\[
\begin{split}
P(\tilde G)  - P(G) &\leq \Ha^n(\pa  B_{r_0/2}(z) \setminus G) - \Ha^n(\pa G \cap  B_{r_0/2}(z) ) \\
&\leq  \frac{2(n+1)|B_1|}{r_0} |B_{r_0/2}(z) \setminus G | \\
&\leq C |\tilde G \Delta E | \leq C \rho^{n+1}. 
\end{split}
\]
and the claim follows. 
\end{proof}

The paper concludes with   Theorem \ref{thm2}. To this aim  I recall the definition of a distributional solution of \eqref{the flow} from \cite{LS}. 
\begin{definition}
\label{def:distributional}
Family of sets of finite perimeter $(E_t)_{t \geq 0}$ is a distributional solution to \eqref{the flow}  starting from $E_0 \subset \R^{n+1}$ if the following holds:
\begin{enumerate}
\item for almost every $t >0$ the set $E_t $ has mean curvature $H_{E_t}$ in a distributional sense and for every $T>0$
\[
\int_{0}^T \|H_{E_t}\|_{L^2(\pa^* E_t)}^2 \, dt  <\infty.
\]

\item There exists $v : \R^{n+1} \times (0,\infty) \to \R$ with $v \in L^2(0,T; L^2(\pa^* E_t))$ such that for every $\phi \in C_0^1(\R^{n+1} \times [0,\infty))$ it holds 
\[ 
\begin{split}
&- \int_0^T  \int_{\pa^* E_t} v \phi \, d \Ha^n \, dt = \int_0^T  \int_{\pa^* E_t}  (H_{E_t} - \bar H_{E_t}) \phi\, d \Ha^n \, dt ,\\
& \int_0^T  \int_{E_t} \pa_t \phi \, dx dt+ \int_{E_0} \phi(\cdot,0)\, dx = - \int_0^T  \int_{\pa^* E_t} v \phi \, d \Ha^n \, dt. 
\end{split}
\]
\end{enumerate}
\end{definition}

\begin{proof}[\textbf{Proof of Theorem \ref{thm2}}]

The proof is exactly the same as \cite[Theorem 2.3]{MSS}. Note that Proposition \ref{coro:minima}  implies that the sets $E_t^h$ are $(Ch^{-1/2}, c\sqrt{h})$-minimizers of the perimeter, i.e., for every $F$ with $F\Delta E_t^h \subset B_{c\sqrt{h}}(x_0)$ it holds
\begin{equation} \label{eq:tekninen-2}
P(E_t^h) \leq P(F) + \frac{C}{\sqrt{h}}| E_t^h \Delta F | .
\end{equation}
\end{proof}

\section*{Acknowledgments}
The author is supported by the Academy of Finland grants  314227 and 347550.

\end{document}